\documentclass{amsart}
\usepackage[utf8]{inputenc}

\usepackage{graphics}
\usepackage{thmtools}
\usepackage{wasysym}
\usepackage[T1]{fontenc}    

\usepackage{amsthm}
\usepackage{amsbsy,amsmath,amssymb,amscd,amsfonts}
\usepackage[pagebackref=true]{hyperref}

\usepackage{graphicx,float,latexsym,color}
\usepackage[font={scriptsize,it}]{caption}
\usepackage{subcaption}

\usepackage{makecell}

\usepackage[dvipsnames]{xcolor}

\newtheorem{theorem}{Theorem}
\newtheorem*{theorem*}{Theorem}

\newtheorem{proposition}{Proposition}

\newtheorem{corollary}{Corollary}
\newtheorem{lemma}{Lemma}
\theoremstyle{remark}
\newtheorem{remark}{Remark}
\newtheorem*{remark*}{Remark}
\theoremstyle{definition}

\hypersetup{
    pdftoolbar=true,        
    pdfmenubar=true,        
    pdffitwindow=false,     
    pdfstartview={FitH},    
    colorlinks=true,       
    linkcolor=OliveGreen,          
    citecolor=blue,        
    filecolor=black,      
    urlcolor=red           
}

\usepackage{lineno}

\usepackage{comment}
\usepackage{microtype}

\title{Area-Invariant Pedal-Like Curves\\Derived from the Ellipse}
\author[D. Reznik]{Dan Reznik} 
\author[R. Garcia]{Ronaldo Garcia}
\author[H. Stachel]{Hellmuth Stachel}
\date{June, 2020}

\begin{document}

\maketitle

\begin{abstract}
We study six pedal-like curves associated with the ellipse  which are area-invariant for pedal points lying on one of two shapes: (i) a circle concentric with the ellipse, or (ii) the ellipse boundary itself. Case (i) is a corollary to properties of the Curvature Centroid (Krümmungs-Schwerpunkt) of a curve, proved by Steiner in 1825. For case (ii) we prove area invariance algebraically. Explicit expressions for all invariant areas are also provided.

\end{abstract}

\section{Introduction}
\label{sec:intro}
Consider an ellipse $\mathcal{E}$ and a fixed point $M$. Let $\mathcal{E}_n$ denote the {\em negative-pedal curve} with respect to $M$ \cite{stachel2019-conics}, i.e., the envelope of lines $\mathcal{L}(t)$ through a point $P(t)$ on $\mathcal{E}$ and perpendicular to $P(t)-M$; see Figure~\ref{fig:npc-3}. This article was motivated by a recent result  \cite{garcia2020-deltoid}: $\mathcal{E}_n$ is a three-cusp area-invariant deltoid for all $M$ on $\mathcal{E}$; see Figure~\ref{fig:npc-3} (top right).

\begin{figure}
    \centering
    \includegraphics[width=\textwidth]{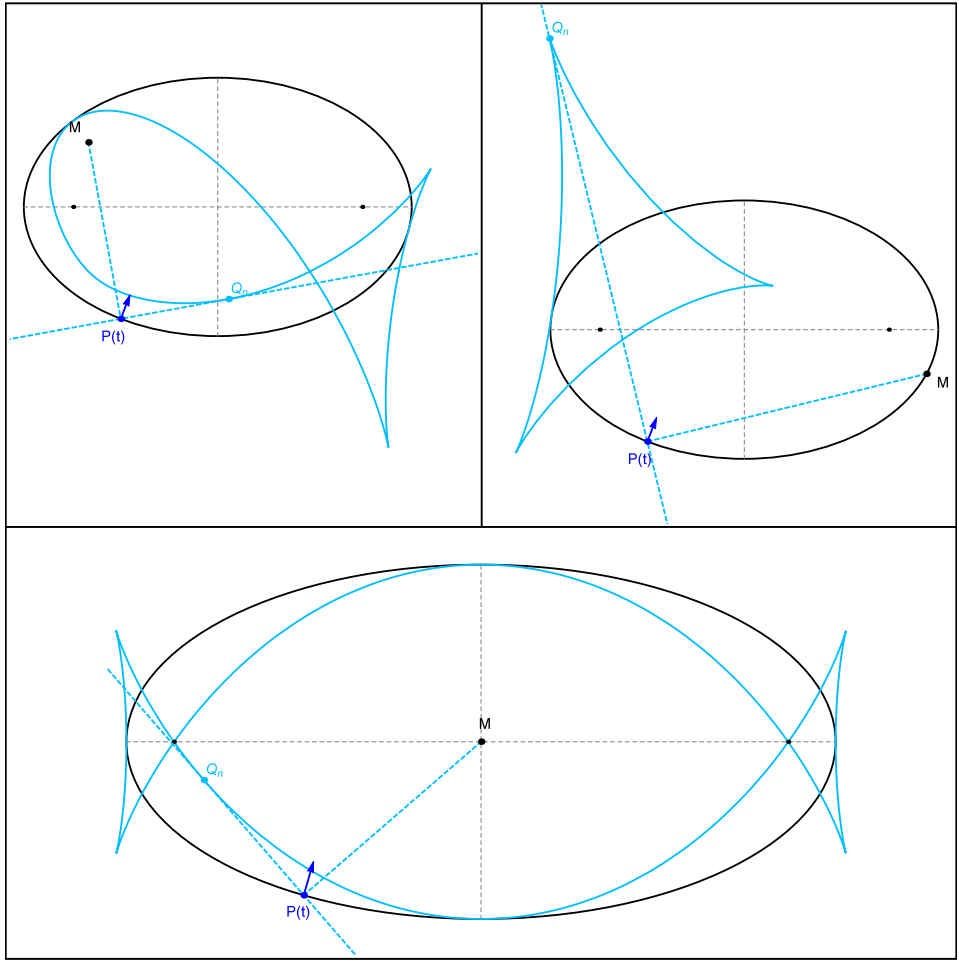}
    \caption{Examples of the Negative-Pedal Curve (light blue) of an ellipse (black) with respect to a point $M$. These are the envelope of lines through $P(t)$ on the ellipse, perpendicular to $P(t)-M$ ($Q_n$ is the tangent point). Three cases are shown, for $M$ (i) interior (top left), (ii) on the boundary (top right), and (iii) at the center (bottom) of the ellipse. For case (ii) the area of the curve is invariant for all $M$ \cite{garcia2020-deltoid}. Case (iii) yields Talbot's Curve \cite{mw} (in general it does not pass through the foci, but for the case shown, $a/b=2$, it does).} 
    \label{fig:npc-3}
\end{figure}

Let $\mathcal{E}_p$, $\mathcal{E}_c$ denote the {\em pedal}, and {\em contrapedal}, curves of $\mathcal{E}$ with respect to a point $M$ \cite{stachel2019-conics}; see Figures~\ref{fig:pedal-cp}. Recall the contrapedal of a plane curve is the pedal of the evolute \cite[Contrapedal]{mw}. For the ellipse, the evolute is a 4-cusp astroid \cite[Ellipse Evolute]{mw}; see Figure~\ref{fig:contrapedal}.  Additionally, define:

\begin{itemize}
    \item The {\em Rotated Pedal Curve} $\mathcal{E}_{\theta}$, the locus of foot $Q_{\theta}$ of a perpendicular dropped from $M$ onto the line through $P(t)$ oriented along a $\theta$-rotated tangent to the ellipse, Figure~\ref{fig:theta-mu}(left).
    \item The {\em Interpolated Pedal Curve} $\mathcal{E}_{\mu}$, the locus a point $Q_{\mu}=(1-\mu)Q_p+{\mu}Q_c$ ($\mu$ is a constant), i.e., an affine combination of pedal and contrapedal feet, Figure~\ref{fig:theta-mu}(right).
    \item The {\em Hybrid Pedal Curve} $\mathcal{E}^*$, the locus of the intersection $Q^*$ of $\mathcal{L}(t)$ with the line from $M$ to $Q_p$, Figure~\ref{fig:hybrid}.
    \item The {\em Pseudo Talbot Curve}\footnote{After the actual Talbot's Curve, shown in Figure~\ref{fig:npc-3} (bottom): the negative pedal curve of an ellipse with respect to its center $O$.} $\mathcal{E}^\dagger$, i.e., the Negative Pedal Curve of $\mathcal{E}^*$, Figure~\ref{fig:hybrid-npc}.
\end{itemize}

\begin{figure}
    \centering
    \includegraphics[width=.7\textwidth]{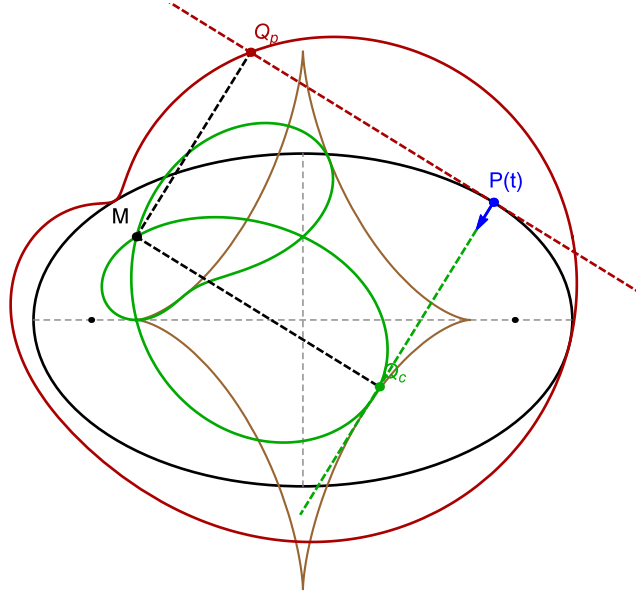}
    \caption{The Ellipse Pedal Curve $\mathcal{E}_p$ (red) is the locus of the foot $Q_p$ of the perpendicular dropped from $M$ onto the line through $P(t)$ tangent to the ellipse. The Contrapedal Curve $\mathcal{E}_c$ (green) is the locus of foot $Q_c$ of the perpendicular dropped from $M$ onto the line through $P(t)$ normal to the ellipse. It can also be regarded as the pedal curve to the ellipse evolute (brown astroid).}
    \label{fig:pedal-cp}
\end{figure}

\begin{figure}
    \centering
    \includegraphics[width=\textwidth]{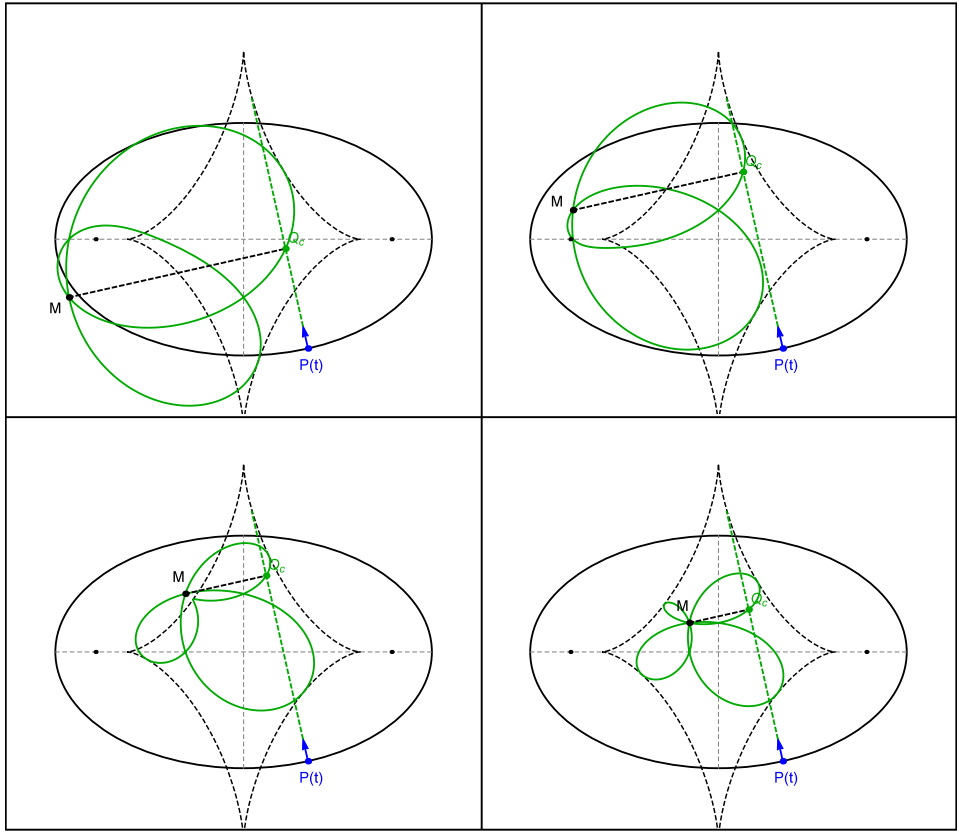}
    \caption{The ellipse Contrapedal Curve $\mathcal{E}_c$ for four distinct positions of $M=[x_m,y_m]$. Notice the contrapedal is the pedal of the evolute (dashed black). We invite the reader to prove that (i) the curve's two self intersections (ignoring the one at $M$) always occur at $[x_m,0]$ and $[0,y_m]$ and (ii) that it touches the evolute at either 2 (top row) or 4 (bottom row) locations.}
    \label{fig:contrapedal}
\end{figure}

\begin{figure}
    \centering
    \includegraphics[width=\textwidth]{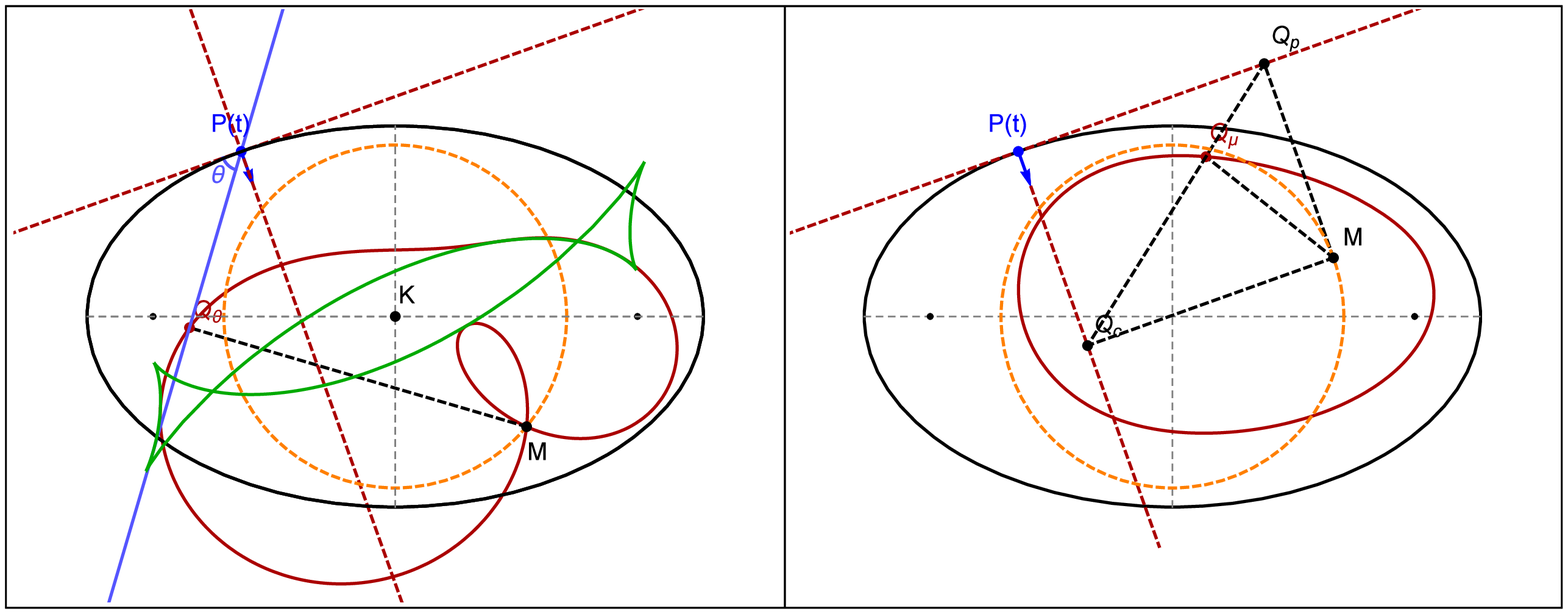}
    \caption{\textbf{Left:} The rotated pedal curve $\mathcal{E}_{\theta}$ (red) is the locus of $Q_{\theta}$, the foot of a perpendicular dropped from $M$ onto a line through $P(t)$, along the $\theta$-rotated tangent (light blue), in this case $\theta=54^\circ$. $A_{\theta}$ is invariant for $M$ on a concentric circle (orange). Note $\mathcal{E}_0=\mathcal{E}_p$ and  $\mathcal{E}_{\pi/2}=\mathcal{E}_c$, and in general, $\mathcal{E}_\theta$ is the pedal curve with respect to the $\theta$-evolutoid (green), whose curvature centroid $K$ is stationary at $O$. \textbf{Right:} The interpolated pedal curve $\mathcal{E}_{\mu}$ (red) is the locus of $Q_{\mu}$, the affine combination of pedal $Q_p$ and contrapedal $Q_c$ feet, here $\mu=1/3$. $A_{\mu}$ is invariant provided $M$ lies on a concentric circle (orange).}
    \label{fig:theta-mu}
\end{figure}

Let $A$, $A_p$, $A_c$, $A_\theta$, $A_\mu$, $A^*$, and $A^\dagger$ denote the areas of $\mathcal{E}$, $\mathcal{E}_p$,
$\mathcal{E}_c$,
$\mathcal{E}_\theta$,
$\mathcal{E}_\mu$,  $\mathcal{E}^*$, and $\mathcal{E}^\dagger$, respectively. 

\subsection*{Main Results}

In Section~\ref{sec:review-steiner} we review a theorem by Jakob Steiner \cite{pamfilos2019-krummungs,steiner1838} concerning the Curvature Centroid (Krümmungs-Schwerpunkt) of polygons; a corollary is that $A_p$, $A_c$, and $A_{\theta}$ are invariant for $M$ along any circle concentric with $\mathcal{E}$. Furthermore, we prove $A_{\mu}$ also shares this property.

In Section~\ref{sec:explicit}, we derive explicit expressions for $A$, $A_p$, $A_c$, $A_\theta$, $A_\mu$ in terms of $\mathcal{E}$'s semi-axes $(a,b)$, $M$, $\mu$ and $\theta$. We also show that (i) $A_p-A_c=A$, and (ii) $A_p-A_\theta=A\sin^2\theta$.

In Section~\ref{sec:main-results} we prove that both $A^*$ and $A^\dagger$ are invariant for $M$ on $\mathcal{E}$.

Appendices~\ref{app:general-evolutoids}, \ref{app:pedal-contrapedal} contain propositions supporting results related to ellipse evolutes, and pedal-like curves, respectively. A table of all symbols used herein appears in Appendix~\ref{app:symbols}.


\section{Sturm and Steiner: Circular Area Isocurves}
\label{sec:review-steiner}
A 1823 Theorem by Sturm states that given a triangle, the area of the {\em pedal triangle} with respect to a point $M$ is constant for all $M$ on a circle centered on the circumcenter $X_3$ \cite[Thm. 7.28, page 221]{ostermann2012}, Figure~\ref{fig:sturm}.

\begin{figure}
    \centering
    \includegraphics[width=\textwidth]{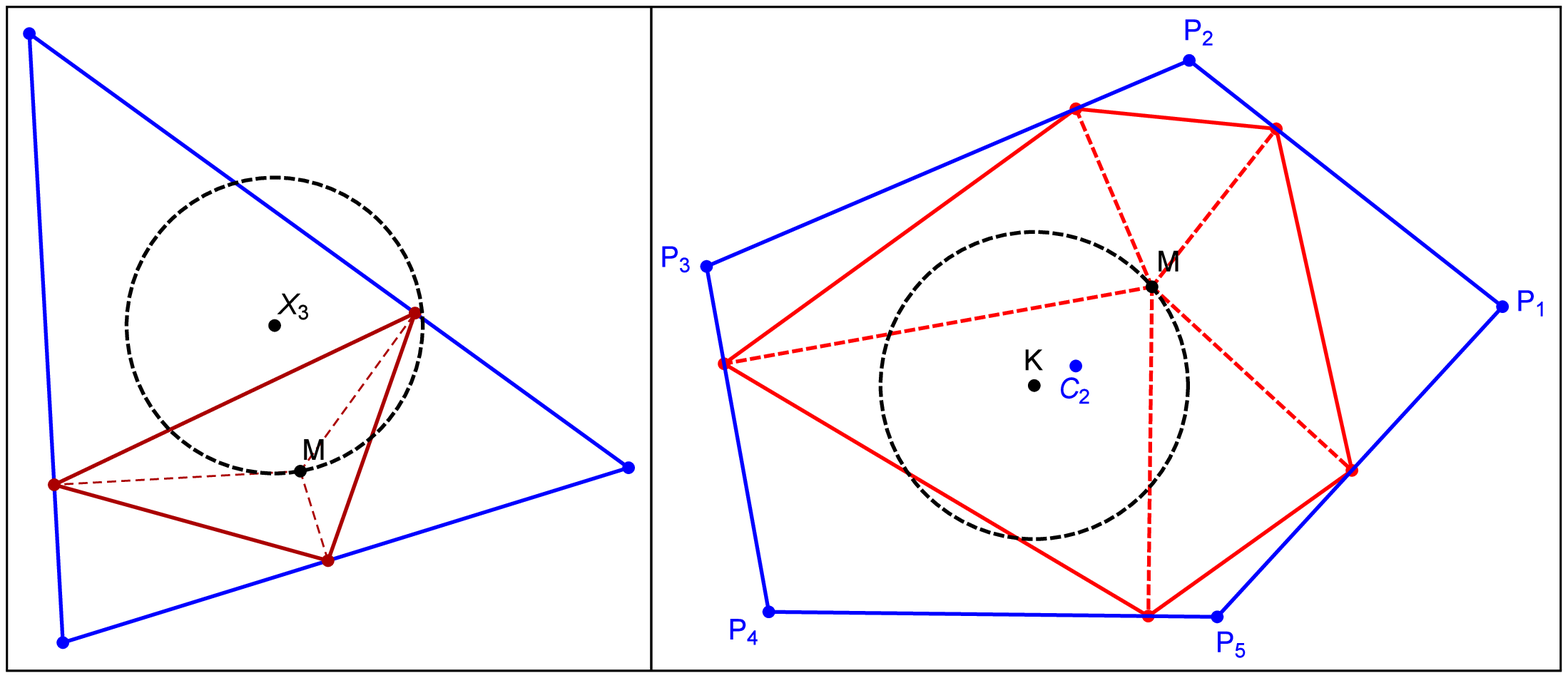}
    \caption{\textbf{Left}: Sturm's Theorem (1823) states that the area of the pedal triangle (red) of a reference triangle (blue) is invariant for all points $M$ lying a circles (dashed black) centered on the circumcenter $X_3$. \textbf{Right}: in 1825, Steiner generalizes this to polygons: the area of the pedal polygon (red) to an N-gon (blue) is constant for all $M$ over a circle centered on $K$, the curvature centroid. $C_2$ denotes the polygon's center of area.}
    \label{fig:sturm}
\end{figure}

In 1825 Steiner generalized it as follows: given a polygon with vertices $P_i,i=1,{\ldots}N$, the area of its pedal polygon with respect to $M$ is invariant for $M$ on a circle centered on Steiner's {\em curvature centroid} $K$, given by \cite{steiner1838}:

\begin{equation}
    K = \frac{\sum_i{\sin(2\theta_i) P_i}}{\sum_i{\sin(2\theta_i)}}
\end{equation}

\noindent where $\theta_i$ are the internal angles, $i=1,\cdots,N$. In the same publication Steiner also proves that the pedal polygon with respect to $K$ has extremal area. Note for $N=3$, $K=X_3$ as the latter has barycentrics of $\sin(2\theta_i)$ \cite{etc}. This is consistent with the fact that pedal polygons with respect to points on the circumcircle have constant area (in fact they have zero area, their vertices lie on the Simson line \cite[Simson Line]{mw}).

Steiner further generalized the above to the case of a closed plane curve $\mathcal{C}$, by approximating it with a polygon where $N{\rightarrow}\infty$. Let the {\em pedal curve} $\mathcal{C}_p$ of $\mathcal{C}$ with respect to a point $M$ be the locus of the foot of the perpendicular dropped from $M$ onto a point $P(t)$ on $\mathcal{C}$ for all $t$; see Figure~\ref{fig:steiner-general}. With $N{\rightarrow}\infty$,
provided that the total curvature of $\mathcal{C}$ is non-zero (i.e., non-zero winding number), $K$ becomes \cite{steiner1838}:

\begin{equation}
    K = \frac{\int{\kappa(s) P(s).ds}}{\int{\kappa(s).ds}}
    \label{eqn:steiner-k}
\end{equation}

\noindent where $\kappa(s)$ is the curvature and $s$ is arc length. Referring to Figure~\ref{fig:steiner-general}, we recall a result by Jakob Steiner \cite{steiner1838}:  

\begin{theorem*}[Steiner, 1825]
The area of the pedal curve is constant over points $M$ lying on circles centered on $K$.
\label{thm:pedal}
\end{theorem*}

\begin{figure}
    \centering
    \includegraphics[width=.5\textwidth]{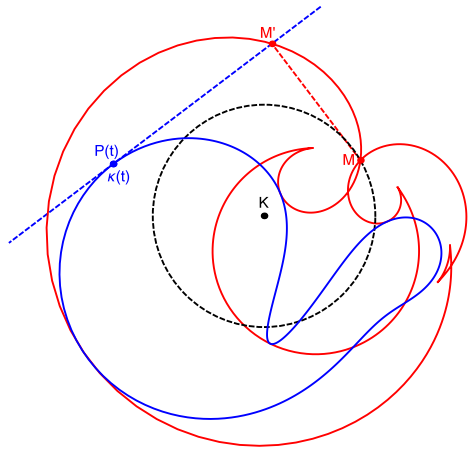}
    \caption{A generic closed curve $\mathcal{C}$ (blue) and its pedal curve (red), defined as the locus of the foot $M'$ of perpendiculars dropped from $M$ onto point $P(t)$ on $\mathcal{C}$. The Steiner curvature centroid $K$ is obtained by averaging the curvature $\kappa(t)$ over all $P(t)$; see equation~\eqref{eqn:steiner-k}. The signed area of the pedal polygon is constant for $M$ over a circle centered on $K$.}
    \label{fig:steiner-general}
\end{figure}

From symmetry:

\begin{lemma}
For the ellipse and its evolute (an astroid), $K=O$.
\label{lem:ellipse-k}
\end{lemma}

Note: when expressed in line coordinates, the cusps of the evolute are regular. Specifically, cusps of the evolute are inflection points of its dual \cite{fischer2001,akopyan2007-conics}. 

Referring to Figure~\ref{fig:pedal-contrapedal}:

\begin{figure}
    \centering
    \includegraphics[width=\textwidth]{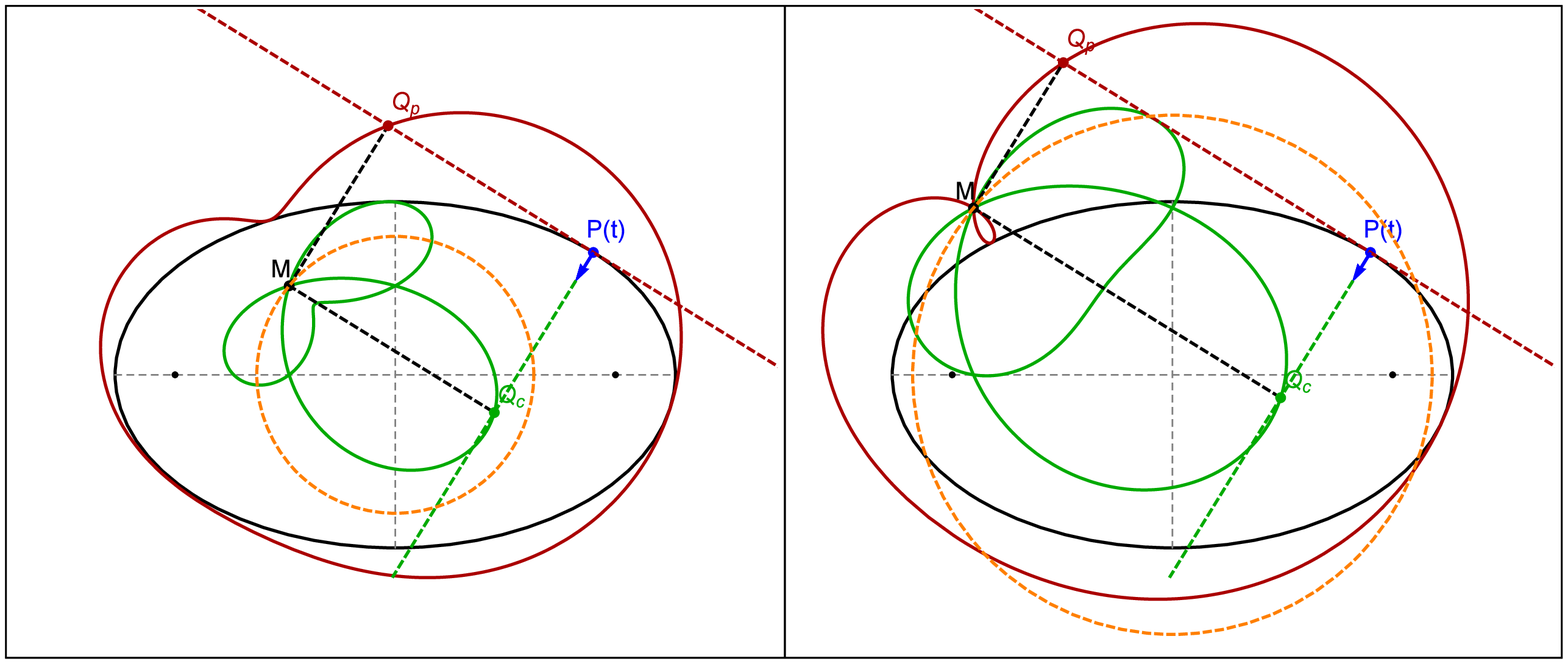}
    \caption{\textbf{Left:} The areas $A_p$ (resp. $A_c$) of the Pedal Curve $\mathcal{E}_{p}$ (red) (resp. Contrapedal Curve $\mathcal{E}_{c}$, green) is invariant over all $M$ on a circle (orange) concentric with the ellipse. \textbf{Right:} An iso-area concentric circle (orange) of radius larger than the minor axis of the ellipse.}
    \label{fig:pedal-contrapedal}
\end{figure}

\begin{corollary}
The area $A_p$ of the pedal curve is invariant for $M$ on a circle concentric with $\mathcal{E}$.
\end{corollary}

As illustrated for an ellipse in Figure~\ref{fig:contrapedal}, in general, the contrapedal curve is the pedal curve with respect to the evolute \cite[Contrapedal Curve]{mw} and:

\begin{corollary}
The area $A_c$ of the contrapedal curve is invariant for $M$ on a circle concentric with $\mathcal{E}$.
\end{corollary}

Let the term $\theta$-evolutoid denote the envelope of $\theta$-rotated tangents to a curve; see Figure~\ref{fig:evolutoid}.

\begin{figure}
    \centering
    \includegraphics[width=\textwidth]{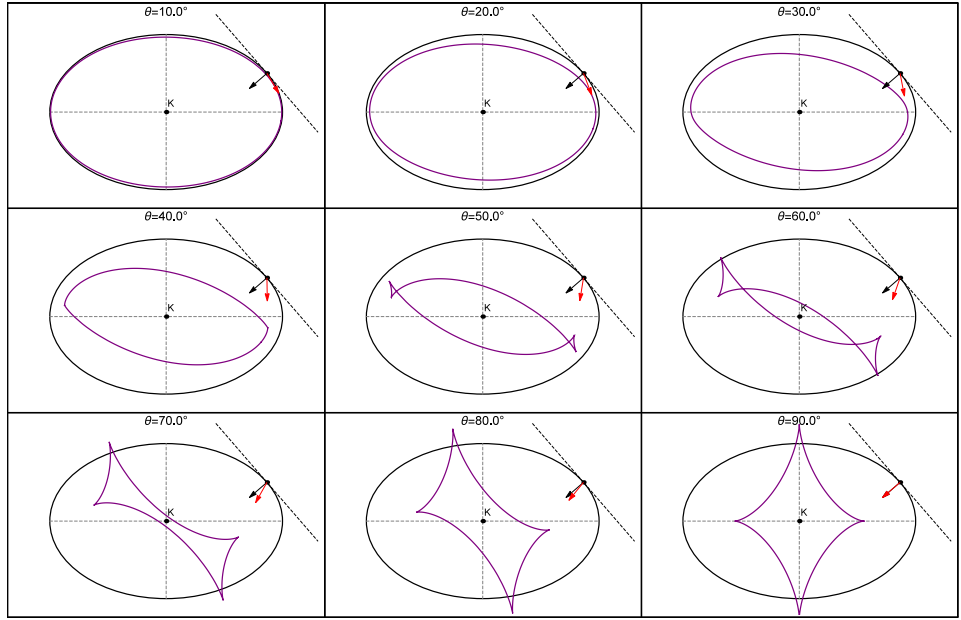}
    \caption{The $\theta$-evolutoid (purple, envelope of tangents rotated by $\theta$) for an $a/b=1.5$ ellipse, $\theta=10,\ldots,90$ degrees. The bottom-right figure ($\theta=90^\circ$) is the ellipse evolute. Also shown are the normal (black arrow) and rotated tangent (red arrow) for a point in the 1st quadrant. Since these curves are centrally symmetric, the Steiner curvature centroid $K$ lies at the ellipse center.}
    \label{fig:evolutoid}
\end{figure}

\begin{lemma}
The $\theta$-evolutoid to an ellipse has $K=O$, for any $\theta$.  
\label{lem:evolutoid-k}
\end{lemma}

This stems from the fact that for all $\theta$, the $\theta$-evolutoid remains symmetric with respect to the origin $O$.

\begin{corollary}
The area $A_{\theta}$ of the rotated contrapedal curve is invariant for $M$ on a circle concentric with $\mathcal{E}$.
\end{corollary}

This stems from the fact that the rotated pedal curve $\mathcal{E}_{\theta}$ is the pedal with respect to a $\theta$-evolutoid and Lemma~\ref{lem:evolutoid-k}.

\begin{theorem}
The area $A_\mu$ of the interpolated pedal curve is invariant for $M$ on a circle concentric with $\mathcal{E}$.
\end{theorem}

\begin{proof}
Proposition~\ref{prop:amu-concave} in Appendix~\ref{app:pedal-contrapedal} shows that for any closed curve with non-zero total curvature (the denominator of Equation~\eqref{eqn:steiner-k}),  $A_{\mu}$ is a fixed linear function of $A_p$ and $A_c$. Since both $A_p$ and $A_c$ are constant for $M$ on circles centered on $K=O$, the result follows.
\end{proof}

\section{Explicit Areas}
\label{sec:explicit}
As before, let a point $P(t)$ on $\mathcal{E}$ be parametrized as $P(t)=(a\cos t,b\sin t)$. Define the {\em signed area} of a curve $\gamma$ as:

\begin{equation}
\mathcal{A}_\gamma=\frac{1}{2}\int_{\gamma}(x{dy}-y{dx}).
\label{eqn:area}
\end{equation}


Referring to Figure~\ref{fig:evolutoid}, the $\theta$-evolutoid is the envelope of lines passing through $P(t)$ rotated with respect to the tangent vector $P'(t)$ by $\theta$. Its coordinates can be derived explicitly as

\begin{align*}
x_{\theta}(t)=& a\cos^{2}{\theta}\cos{t}
  + \frac {c^2 \sin^2  \theta 
		\cos^3   t  
	 }{a} - \frac {\sin t \sin   \theta
	  \cos  \theta    (   b^2\cos^2 t
		 +{a}^{2}   \sin^2 t )}{b} 
 \\
y_{\theta}(t)=&a\sin   \theta
\cos  \theta  \cos   t - \frac {c^2\sin{\theta} \cos^{2}t \left( b\cos
		t \cos  \theta  -a\sin  \theta
	  \sin t \right) }{ab}  \\
  +& \frac {
		\sin t   \left( b^2\cos^2   \theta 
	  -c^2 \sin^2  \theta  
		 \right)}{b}
\end{align*}

\noindent with $c^2=a^2-b^2$.

Let $\theta_0=\tan^{-1}\left(\frac{2ab}{{3}c^2}\right)$.

\begin{remark}
The $\theta$-evolutoid will have 4, 2, or 0 singularities if
$\theta\in(\theta_0,\pi-\theta_0)$, $\theta\in\{\theta_0,\pi-\theta_0\}$, or $\theta\notin[\theta_0,\pi-\theta_0]$, respectively. Moreover, the ${\theta_0}$-evolutoid is singular at $t_1=\frac{3\pi}{4}$ and
	$t_2=  \frac{7\pi}{4}$.
	
\end{remark}

\begin{proposition}\label{prop:aelipse}
The signed area $S_\theta$ of the $\theta$-evolutoid is given by:

\[ S_{\theta}=\pi\,ab   \cos^2   \theta  -  \,{\frac {3
		    c^{2} }{8ab}}\sin^2  \theta   
\]
 
\end{proposition}
\begin{proof}
Direct integration of Equation~\ref{eqn:area}.
\end{proof}

Let $M = (x_0,y_0)$. 

\begin{proposition}\label{prop:areapc}
The areas $A_p$ and $A_c$ of $\mathcal{E}_p$ and  $\mathcal{E}_c$ are given by: 

\begin{align}
A_p=&\frac{\pi}{2}\, \left( {a}^{2}+{b}^{2}+   x_0 ^{2}+  y_0^{2} \right) \label{eqn:ap} \\
A_c=&\frac{\pi}{2}\, \left( {a}^{2}+{b}^{2}-2ab+   x_0 ^{2}+  y_0^{2} \right) \nonumber
\end{align}
\end{proposition}

\begin{proof} Consider the ellipse parametrized by $P(t)=(a\cos t,b\sin t)$. Then it follows that
{\small 
\[\aligned
	\mathcal{E}_p(t)=& \left[\frac{(  a^2 x_0\sin^2 t -ab y_0\cos t\sin t+ab^2\cos t)}{   b^2\cos ^2t+a^2\sin ^2t },  \frac{(b^2y_0\cos^2t -abx_0\cos t \sin t +a^2b\sin  t}{   b^2 \cos ^2t +a^2 \sin^2 t}\right]\\
	\mathcal{E}_c(t)=&\left[ \frac{b^2x_0\cos ^2t  + \cos t\sin t (aby_0+ac^2\sin{t})}{ b^2 \cos ^2t +a^2 \sin^2 t},
	\frac{  a^2 y_0\sin ^2t + \cos t\sin t(abx_0-b c^2\cos{t})  }{ b^2 \cos ^2t +a^2 \sin^2 t}\right]
	\endaligned \]
	}
Compute the above areas with Equation~\ref{eqn:area}. The integrand will be a ratio of trigonometric polynomials. Evaluate the integrals by using classical residue theory \cite{ahlfors1979-complex}. Algebraic manipulation yields the claim.
\end{proof}

\noindent Note: formulas in \eqref{eqn:ap} and later are consistent with Steiner's result that the area of the pedal curve of $M$ is the sum of the area for $M=K$ and a term proportional to the square of $|MK|$ \cite[p.~47]{steiner1838}.

\begin{corollary}
$A_p-A_c=A$.
\label{cor:area-diff}
\end{corollary}

\noindent Note the above holds holds for all convex curves, as proved in Proposition~\ref{prop:darea}, Appendix~\ref{app:pedal-contrapedal}.

\begin{proposition}
The area $A_\theta$ of the rotated pedal curve is given by

\[ A_\theta=\frac{\pi}{2}\, \left(     {a}^{2}+b^2-2\,ab\sin^2\theta + x_0^{2}+y_0^{2} \right) \]
\end{proposition}

\begin{proof}
Similar to Proposition \ref{prop:areapc}.
\end{proof}

\begin{corollary}
$A_p-A_\theta=A\sin^2\theta$.
\end{corollary}

\begin{proposition}
\label{prop:interpol}
The area $A_\mu$ of $\mathcal{E}_{\mu}=(1-\mu)\mathcal{E}_p+\mu \mathcal{E}_c$ is given by

\begin{align*} A_{\mu}=& \left(2\,{a}^{2}-3\,ab+2\,b^{2}+2\,x_0^{2}+2\,y_0^{2} \right) \pi\,{\mu}^2 -2\,\left( (a-b)^2  +x_0^{2}
+y_0^{2} \right) \pi\,\mu \\
+&\frac{1}{2}\left( (a-b)^2 +x_0^{2}+y_0^{2} \right) \pi\\
  =&(2\mu-1)[(1-\mu) A_p-\mu A_c]+\mu(1-\mu)A. 
\end{align*}
\end{proposition}

\begin{proof}
Similar to Proposition \ref{prop:areapc}. 
\end{proof}

\section{Area Invariance of Hybrid and Pseudo Talbot Curves}
\label{sec:main-results}
As defined in Section~\ref{sec:intro}, let (i) the Hybrid Pedal Curve $\mathcal{E}^*$ be the locus of the intersection $Q^*$ of $\mathcal{L}(t)$ with the line from $M$ to $Q_p$, Figure~\ref{fig:hybrid}, and (ii) the Pseudo Talbot Curve $\mathcal{E}^\dagger$ be  the Negative Pedal Curve of $\mathcal{E}^*$, Figure~\ref{fig:hybrid-npc}. Here we prove their area invariance over all $M$ on $\mathcal{E}$.

\begin{theorem}
	The area $A^*$ of $\mathcal{E}^*$ is invariant for all $M$ on $\mathcal{E}$ and given by:
	\[ A^*= \frac{\pi  (3a^4+2a^2b^2+3b^4)}{2 a b}=\frac{\pi(3\delta^2+5a^2b^2)}{2 a b},\;\; \delta=\sqrt{a^4-a^2b^2+b^4}
	\]
\label{thm:area-hybrid}
\end{theorem}

\begin{proof}

Let $P(t)=[a\cos{t},b\sin{t}]$ be a point on $\mathcal{E}$. By definition (Section~\ref{sec:intro} and  Figure~\ref{fig:hybrid}), $\mathcal{E}^*$ is defined as the intersection of lines
\[ M+ uP'(t)^{\perp}\;\;\; \text{and}\;\;\; P(t)+v (M-P(t))^{\perp}.\]

\noindent let $M=(x_0,y_0)$ and $\mathcal{E}(t)=[x^*(t),y^*(t)]$. Straightforward calculation leads to:

{\small
\begin{align*}
x^*(t)=&\frac{
-b\left(3{a}^{2}+{b}^{2}+4{y_0}^{2}\right)\cos{t}
+4{a}{b}{x_0}\cos{2{t}}
-b c^2 \cos{3{t}}
+4{a}{x_0}{y_0}\sin{t}
+4{b}^{2}{y_0}\sin{2{t}}
}{
4\left(a{y_0}\sin{t}+b{x_0}\cos{t}-{a}{b}\right)
}\\
y^*(t)=&\frac{
-a\left({a}^{2}+3{b}^{2}+4{x_0^{2}}\right)\sin{t}
+4{a}^{2}{x_0}\sin{2{t}}
-a c^2\sin{3{t}}
+4{b}{x_0}{y_0}\cos{t}
-4{a}{b}{y_0}\cos{2{t}}
}{
4\left(a{y_0}\sin{t}+b{x_0}\cos{t}-{a}{b}\right)
}
\end{align*}
}

\noindent where $c^2=a^2-b^2$. Integrating Equation~\ref{eqn:area} over $\mathcal{E}^*(t)$ yields the claim.
\end{proof}

\begin{figure}
    \centering
    \includegraphics[width=\textwidth]{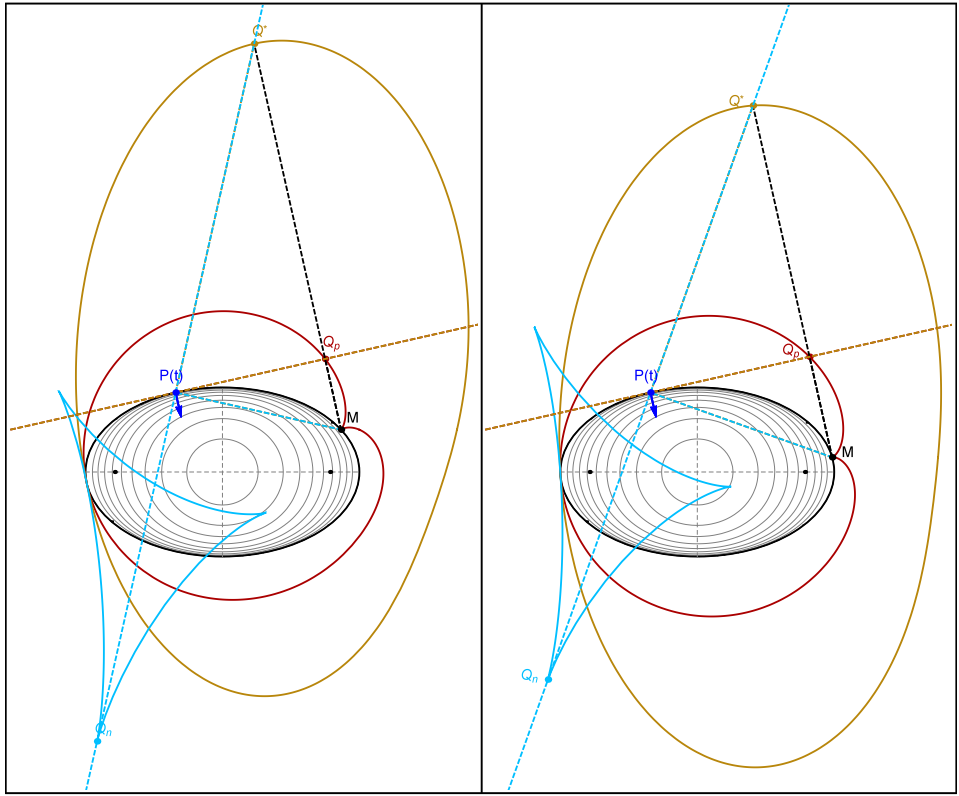}
    \caption{The Pedal $\mathcal{E}_p$ (red) and Hybrid Pedal $\mathcal{E}^*$ (light brown), and Negative-Pedal $\mathcal{E}_n$ (light blue) Curves of the ellipse $\mathcal{E}$ (black) shown for two positions of $M$. $\mathcal{E}^*$ is the locus of the intersection $Q^*$ of (i) the line through $P(t)$ perpendicular to $P(t)-M$ and (ii) line $M{Q_p}$ (note $\mathcal{E}_p$ is the locus of $Q_p$). For both left and right pictures (indeed for all $M$ on the ellipse), areas $A_n,A^*$ are constant, thought that of $\mathcal{E}_p$ varies. Also shown are iso-curves of constant $A^*$ inside the ellipse; these are high-order rational curves. $\mathcal{E}^*$ is unstable when $M$ is exterior to $\mathcal{E}$.}
    \label{fig:hybrid}
\end{figure}

\begin{theorem}
The area $A^\dagger$ of $\mathcal{E}^\dagger$ is invariant for all $M$ on $\mathcal{E}$ and given by:
	
	\[ A^\dagger=    \frac {\pi\, \left( 3\,{a}^{4}+2\,{a}^{2}{b}^{2}+3\,{b}^{4}
 \right)  \left( {a}^{2}-2\,ab-{b}^{2} \right)  \left( {a}^{2}+2\,ab-{
b}^{2} \right) }{8 a^{3} b^{3}}
	\]
	\label{thm:area-talbot}
\end{theorem}

\begin{proof}
Recall $\mathcal{E}^\dagger$ is the negative pedal curve of $\mathcal{E}^*$ (Section~\ref{sec:intro} and Figure~\ref{fig:hybrid-npc}). For $M$ on the ellipse, the coordinates $[x^\dagger,y^\dagger]$ of $\mathcal{E}^\dagger$ can be derived explicitly:

{\scriptsize
\begin{align*}
x^{\dagger}(u)=&-{\frac { \left( \left(  k_x+1 \right) {a}^{4}-2
 \left(k_x+1\right) {b}^{2}{a}^{2}+ k_x b^4  \right)\cos u
}{{b}^{2}a}}\\
-&2\,{\frac { \left( {a}^{2}-{b}^{2} \right) ^{2}   
\sin^{3}t \cos{t}\sin{u}}{{b}^{2}a}} 
 + \frac {\cos t \left( {a}^{2}+{b
}^{2} \right)  \left( -{a}^{2} \sin^2t-{b}^{2} \cos^2t+2\,{b}^{2}
 \right) }{{b}^{2}a}\\
 y^{\dagger}(u)=&-2\,{\frac { \left( {a}^{2}-{b}^{2} \right) ^{2}\sin t
 \cos^3{t}  \cos u}{{a}
^{2}b}}-{\frac { \left(  \left(k_y-1 \right) {a
}^{4}-2 k_y {b}^{2}{a}^{2}+
 k_y {b}^{4} \right) \sin u}{{a}^{2}b}}\\
 +&{\frac {\sin t \left( {a}^{2}+{b
}^{2} \right)  \left( ({a}^{2}-b^2) \cos^2t+{a}^{2} \right) 
}{{a}^{2}b}}\\
k_x=&2\cos^4t-3\, \cos^2t\\
k_y=&2\cos^{4}t- \cos^2t
\end{align*}
}

Integrating Equation~\eqref{eqn:area} for the above yields the claimed results.
\end{proof}

\begin{figure}
    \centering
    \includegraphics[width=\textwidth]{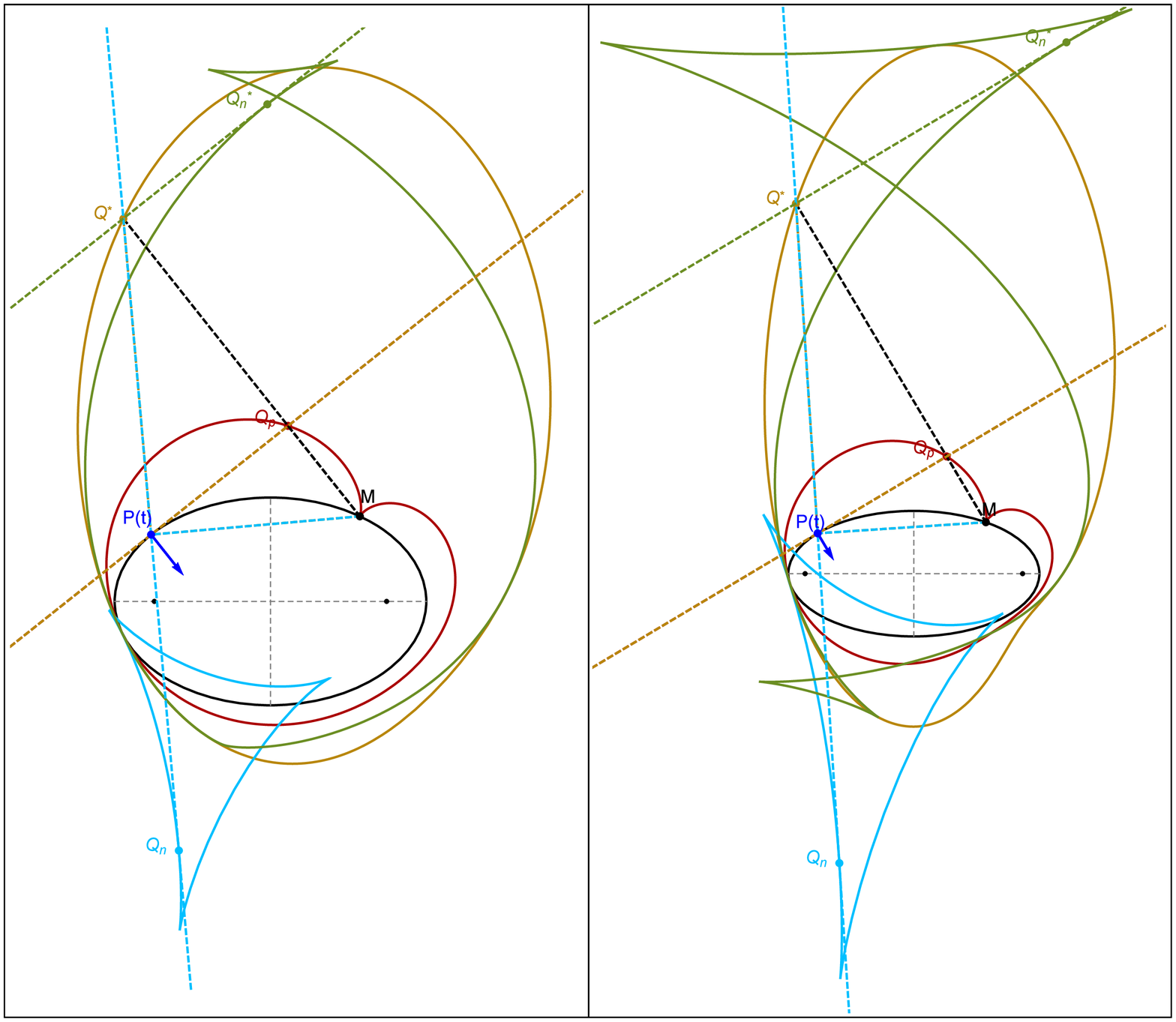}
    \caption{The pedal, negative pedal, hybrid pedal $\mathcal{E}^*$, and pseudo-Talbot curve $\mathcal{E}^\dagger$ (negative pedal curve of $\mathcal{E}^*$) are shown red, light blue, brown, and olive green, respectively, for aspect ratios $a/b$ of $\mathcal{E}$ of 1.5 (left) and 2.0 (right), respectively. Notice that for the former case $\mathcal{E}^\dagger$ has two cusps, and in the latter 4. Excluding the pedal curve, all other 3 are area-invariant over $M$ on $\mathcal{E}$, and are all tangent to $\mathcal{E}$ at the point where the normal goes thru $M$.}
    \label{fig:hybrid-npc}
\end{figure}

\section{Conclusion}
\label{sec:conclusion}
One open question is whether a common thread exists which links the Steiner Hat \cite{garcia2020-deltoid}, the Hybrid, and Pseudo-Talbot curves, since all of them are area-invariant over $M$ on the ellipse. Furthermore, if a continuous family of curves exists with this area-invariance property. 

\section*{Acknowledgments}
We would like to thank Robert Ferréol and Mark Helman for their help during this work.

The second author is fellow of CNPq and coordinator of Project PRONEX/ CNPq/ FAPEG 2017 10 26 7000 508.

\appendix

\section{Evolutoids}
Consider a plane convex curve $\mathcal{C}(t)=(x(t),y(t))$ defined by a support function $h$:

\begin{align}
x(t)=&h(t)\cos t-h'(t)\sin t \label{eqn:support}\\
y(t)=&h(t)\sin t+h'(t)\cos t \nonumber
\end{align}

The family of lines passing through $P(t)=(x(t),y(t))$ making a constant angle $\theta$ with $(x^\prime(t),y^\prime(t)) $ is given by:

\begin{align*} 
\mathcal{L}_{\theta}(t):& \left( \sin \left( 2\,\theta \right) +\sin \left( 2\,t \right) 
\right) x+ \left( \cos \left( 2\,\theta \right) -\cos \left( 2\,t
\right)  \right) y\\
-&h(t)  (  \sin t - 
 \sin \left( 2\,\theta+t \right) )+  h '(t)   (  \cos t -\cos
\left( 2\,\theta+t \right) )=0
\end{align*}

Let $\mathcal{C}_\theta(t)=(x_\theta,y_\theta)$ denote the envelope of $\mathcal{L}_\theta(t)$. This will be given by:
  
\begin{align*} x_{\theta}(t)=&\frac{1}{2}  \left( \cos \left( t-2\,\theta \right) + 
  \,\cos t  \right) h \left( t \right) -h'(t)\sin t    +\frac{1}{2} \left(\cos \left( t-2\,\theta 
  \right) - \,\cos t  \right)  h'' \left( t \right) \\
  y_{\theta}(t)=&\frac{1}{2} \left(\sin \left( t-2\,\theta \right)  +
  \,\sin t \right) h(t)   +h'(t)\cos t    +\frac{1}{2} \left(\sin \left( t-2\,\theta 
  \right) - \,\sin t  \right)  h'' (t)
 \end{align*}
 
\noindent Note that $\mathcal{C}_{\pi/2}$ is the {\em evolute} of $\mathcal{C}$. Let $h_{\theta}(t)=h(t-\theta)\cos\theta +h'(t-\theta)\sin\theta$. Changing variables $t=t-\theta$ it follows that the envelope is  given by
  
\begin{align*}
  x_{\theta}(t-\theta)=&h_{\theta}(t)\cos t-h_{\theta}'(t)\sin t\\
  y_{\theta}(t-\theta)=&h_{\theta}(t)\sin t+h_{\theta}'(t)\cos t
\end{align*}
  
Let $S(.)$ denote the signed area of a curve. Then
  
  \begin{align*}
  S(\mathcal{C})=&\frac{1}{2} \int_0^{2\pi}( h(t)^2-h'(t)^2) dt\\
  S(\mathcal{C}_{\pi/2})=&\frac{1}{2} \int_0^{2\pi}( h'(t)^2-h''(t)^2) dt
  \end{align*}
  
 \begin{proposition}\label{prop:area}
 $S(\mathcal{C}_\theta)$ is given by
  	
  	\[ S(\mathcal{C}_\theta)= S(\mathcal{C})\cos^2\theta+S(\mathcal{C}_{\pi/2})\sin^2\theta\]

  \end{proposition}

\begin{proof}
The signed area of the evolute $\mathcal{C}_{\pi/2}$ is negative in general, and zero if $\mathcal{C}$ is a circle. Integrating Equation~\ref{eqn:area} by parts and simplifying it yields the claim.
\end{proof}


Let $L(.)$ denote the perimeter of a curve.

\begin{proposition}
For small $\theta$, $L(\mathcal{C}_{\theta}$) is given by:
	
	\[L(\mathcal{C}_{\theta} )=L(\mathcal{C})\cos{  \theta} \]
\end{proposition}

\begin{proof} Let $T,N$ define the tangent and normal axis of the Frenet frame. From \cite{giblin_2014} we have that
	\[  \mathcal{C}_{\theta}(s)=\mathcal{C}(s)+\frac{\cos\theta\sin\theta}{k(s)}T(s)+\frac{ \sin^2\theta}{k(s)}N(s).\]
	
Differentiating the above and using Frenet equations $T'={k}N$ and $N'=-{k}T$, it follows that
	
\[ \mathcal{C}_{\theta}'(s)=\frac{ k(s)  ^{2}\cos{\theta} -  k'(s)\sin{\theta}}{ k(s) ^{2}}  \left( N(s) \sin{\theta}+T(s) \cos{\theta} \right) 
	\]
	Therefore,
	\[ |\mathcal{C}_{\theta}'(s)|= |\cos\theta -\frac{k'(s)}{k(s)^2}\sin\theta| .\]
	Integration leads to the result stated.  
\end{proof}

\label{app:general-evolutoids}

\section{Pedal and Contrapedal Areas}
Let $M=(x_0,y_0)$ be a fixed point. Referring to Equation~\ref{eqn:support}, the pedal of $M$ is given by

\begin{equation}\label{eq:pedal}
 P_M(t)= [  x_0 \sin^2 t + \left( h(t) -y_0\sin
 t    \right) \cos t, \left( h(t) -{  x_0}\,\cos t  
   \right) \sin t +   y_0\cos^2 t ].
\end{equation}

The contrapedal of $M$ is given by

\begin{equation}\label{eq:contrapedal}
C_M(t)=  [ x_0  \cos^2t   +y_0\cos t
 \sin t -h'(t) \sin t , y_0 \sin^2t
 +x_0\cos t \sin t 
	 + h'(t) \cos t]
.
\end{equation}

Below is a generalization of Corollary~\ref{cor:area-diff}.

\begin{proposition}
For all convex curves, the following holds:
	\[A(P_M)-A(C_M)=A(\mathcal{C})\]
	\label{prop:darea}

\end{proposition}

\begin{proof}
Obtain the signed areas for the above curves above via integration by parts of  Equation~\ref{eqn:area}. Algebraic manipulation yields the claim.
In fact, 
\begin{align*}
A(P_M)=& \frac{\pi}{2}(x_0^2+y_0^2)- \left(\int_0^{2\pi}\!\!\!\!\!\!h(t)\cos t dt\right) x_0- \left(\int_0^{2\pi}\!\!\!\!\!\!h(t)\sin t dt \right) y_0+\frac{1}{2}  \int_0^{2\pi}\!\!\!\!\!\!h(t)^2 dt\\
A(C_M)=& \frac{\pi}{2}(x_0^2+y_0^2)- \left(\int_0^{2\pi}\!\!\!\!\!\!{h(t)\cos{t}dt}\right) x_0- \left(\int_0^{2\pi}\!\!\!\!\!\!h(t)\sin t dt \right) y_0+\frac{1}{2}\int_0^{2\pi}\!\!\!\!\!\!h^\prime(t)^2 dt\\
A(\mathcal{C})=& \frac{1}{2}\int_0^{2\pi}\!\!\!\!\!(h(t)^2-h^{\prime}(t)^2)dt
\end{align*}
\end{proof}


\begin{corollary} The family of isocurves of $A(P_M)$ and $A(C_M)$ are circles centered at
	
	\begin{align*}
	\label{eq:centroK}
K=\left(	\frac{1}{\pi} \int_0^{2\pi}\!\!\!\!\! h(t)\cos t \,dt,	\frac{1}{\pi} \int_0^{2\pi}\!\!\!\!\! h(t)\sin t \, dt\right)
	\end{align*}

\end{corollary}

\begin{proof} Direct from the definition of the centroid  $K=\frac{1}{2\pi}\int_0^{2\pi}\mathcal{C}(t)dt$ and expressions of the areas $A(P_M)$ and $A(C_M)$.
\end{proof}

\begin{proposition} Let $P_{ {\theta }M}$ the pedal of $M$ with respect to the curve $\mathcal{C}_\theta$.  For any convex curve $\mathcal{C}$ we have that:
	\[A(P_M)-A(P_{ {\theta }M})=\sin^2\theta A(\mathcal{C})\]
	
	\end{proposition}

\begin{proof} Similar to the that of Proposition \ref{prop:darea}.
\end{proof}

Referring to Figure~\ref{fig:interp-concave} and generalizing Proposition~\ref{prop:interpol}:

	
	

	

	

\begin{proposition}
\label{prop:amu-concave}
For any smooth regular closed curve $\mathcal{C}$ with non-zero rotating index, the isocurves of $A(\mathcal{C}_\mu)$ are circles centered on $K$.
In fact, \[A(\mathcal{C}_\mu) =(2\mu-1)[(1-\mu) A(P_M)-\mu A(C_M)]+\mu(1-\mu)A(\mathcal{C}).\]
\end{proposition}

\begin{proof}
Consider a regular closed curve $\mathcal{C}(s)=(x(s),y(s))$ parametrized by arc length $s$ and of length $L$. Let   $M=(x_0,y_0)$. Write $\mathcal{C}^\prime(s)=(\cos\theta(s),\sin\theta(s)).$ Therefore, the curvature is  $k(s)=\theta^\prime(s)$.

Then, the pedal $P_M$ and contrapedal $C_M$ curves with respect ot $M$ are given by
{\small
\begin{align*}
    P_M=& [ (x_0-x(s))\cos^2\theta(s) +\sin\theta(s)\cos \theta(s) (y_0-y(s) )+x(s),\\
    &(y(s)-y_0 )\cos^2\theta(s) +\sin\theta(s)\cos\theta(s) (x_0-x(s))+y_0 ]\\
    C_M=& [x_0+(x(s)-x_0)\cos^2\theta(s)+(y(s)-y_0)\cos\theta(s)\sin\theta(s) ,
    \\&y(s)+(y_0-y(s))\cos^2\theta(s)+(x(s)-x_0)\cos\theta(s)\sin\theta(s). ]
\end{align*}
}
Then,
{\tiny
\begin{align*}
A(P_M)=& \frac{1}{4}\int_0^L k(s)ds\left(x_0^2+ y_0^2\right)- \frac{1}{2}\left(\int_0^L k(s)x(s) ds \right) x_0 -\frac{1}{2}\left(\int_0^L k(s)y(s) ds \right)y_0 \\
+&\frac{1}{2}\int_0^L[\cos\theta(s) y(s)-\sin\theta(s)x(s)]^2 k(s)ds
\\
A(C_M)=& \frac{1}{4}\int_0^L k(s)ds\left(x_0^2+ y_0^2\right)- \frac{1}{2}\left(\int_0^L k(s)x(s) ds \right) x_0 -\frac{1}{2}\left(\int_0^L k(s)y(s) ds \right)y_0 \\
+&\frac{1}{2}\int_0^L[\cos\theta(s) x(s)+\sin\theta(s)y(s)]^2 k(s)ds
\end{align*}
  Therefore, \begin{align*}
      A(P_M)-A(C_M)= &   -\frac{1}{2} \int_0^L    \left( 2\,x \left( s \right) y \left( s \right) \sin   2\,
\theta \left( s \right)    + (    x \left( s \right)  ^{2}  -  y \left( s
  \right) ^{2}) \cos   2\,\theta \left( s
 \right)   \right)   
 k(s)ds  \\
  A(P_M)+A(C_M)= &   \frac{1}{2}\int_0^L k(s)ds\left(x_0^2+ y_0^2\right)-\left(\int_0^L k(s)x(s) ds \right) x_0 - \left(\int_0^L k(s)y(s) ds \right)y_0\\
  +&\frac{1}{2}\int_0^L (x(s)^2+y(s)^2)k(s)ds
  \end{align*}
 }
Let $\mathcal{C}_\mu=(1-\mu) P_M+{\mu}C_M$.

Then,
{  \tiny
\begin{align*}
    A(\mathcal{C}_\mu)=   &(2\mu-1)^2\left[ \left( \frac{1}{4}  \int_0^L k(s)ds \right) (x_0^2+y_0^2) -  \left(\frac{1}{2} \int_0^L k(s)  x(s)ds\right)   x_0  - \left(\frac{1}{2} \int_0^L k(s)  y(s)ds\right)  y_0\right]  \\
      +&\frac{(2\mu-1)}{4}  \int_0^L \left[ 2 x(s)y(s)\sin2\theta(s)+\cos2\theta(s)(x(s)^2-y(s)^2) \right] k(s)ds\\
     +&\frac{(2\mu-1)^2}{4} \int_0^L    (x(s)^2+y(s)^2)  k(s)ds 
     +  \int_0^L \frac{1}{2}\mu(\mu-1)\left( \cos\theta(s) y(s)-x(s)\sin\theta(s) \right) ds\\
     =&(2\mu-1)[(1-\mu) A(P_M)-\mu A(C_M)]+\mu(1-\mu)A(\mathcal{C}).
\end{align*}
}	
\end{proof}

\begin{remark*} When $\int_0^L k(s)ds=0$, or equivalently the rotating index of the curve is zero, the Steiner curvature centroid $K$ is not defined. In this case the pedal and contrapedal area isocurves will be either parallel lines or independent of $M$.
\end{remark*}

\begin{figure}
    \centering
    \includegraphics[width=\textwidth]{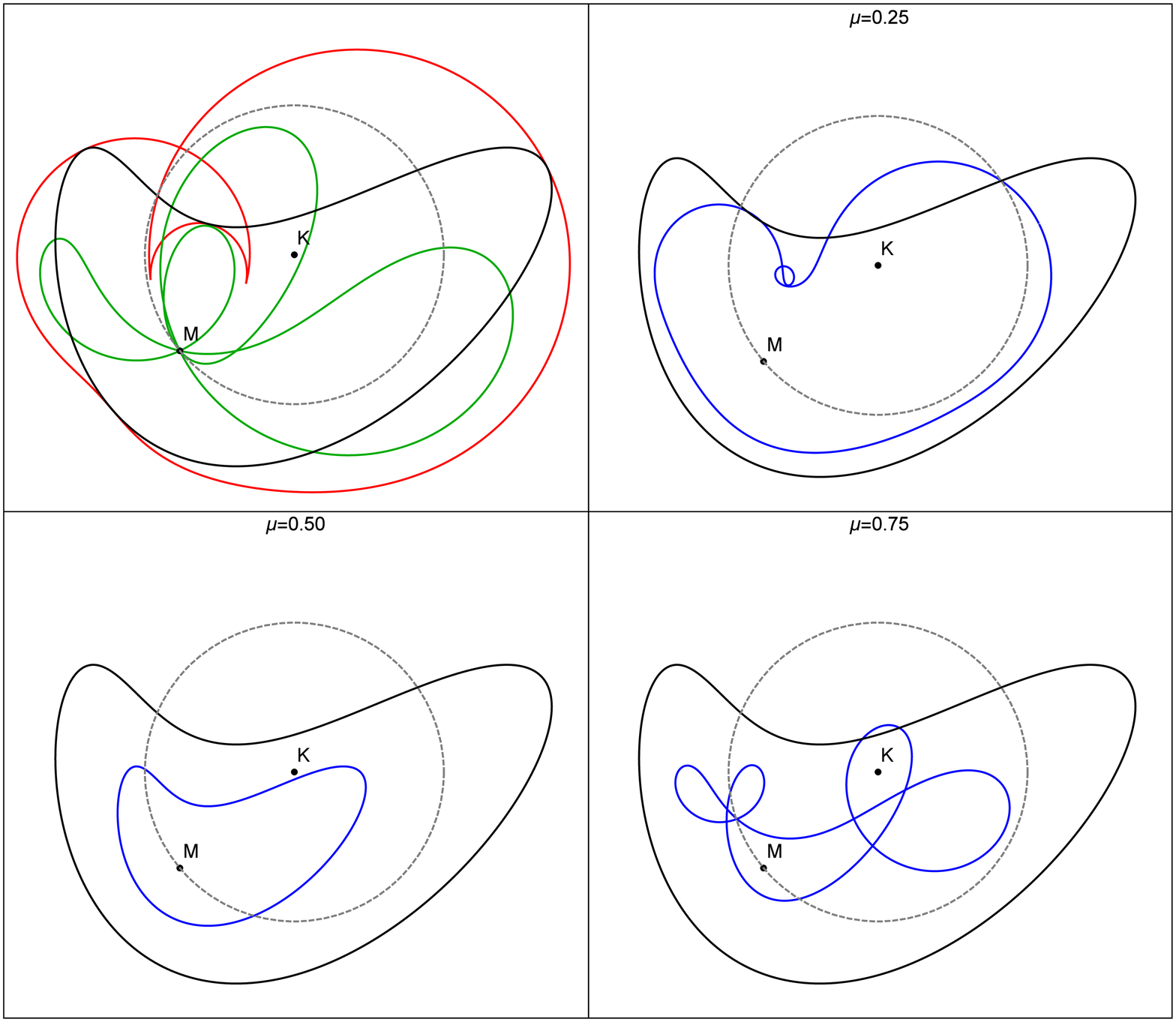}
    \caption{\textbf{Top left}: a generic concave curve $\mathcal{C}$ (black) is shown as well as its curvature centroid $K$ and a circle about it where $M$ lies. Also shown are the pedal (red) and contrapedal (green) curves with respect to $M$. Note their areas are invariant for $M$ anywhere on a circle centered on $K$. \textbf{Top right, bottom left, bottom right}: the interpolated pedal curve (blue) for $\mu=0.25$, $0.5$, and $0.75$, respectively. Its area is also invariant for $M$ on a circle centered on $K$. Notice that at $\mu=0.5$ (bottom left) the interpolated pedal is homothetic (scale of 1/2) to $\mathcal{C}$ and its shape (and area) is independent of the location of $M$.}
    \label{fig:interp-concave}
\end{figure}
\label{app:pedal-contrapedal}


\section{Table of Symbols}
\begin{table}[!htbp]
\small
\begin{tabular}{|c|l|l|}
\hline
symbol & meaning & note \\
\hline
$\mathcal{E}$ & ellipse & semi-axes $a,b$\\
$\mathcal{K}_r$ & circle of radius $r$ concentric with $\mathcal{E}$ & \\
$M$ & a point in the plane & \\
$P(t)$ & a point on $\mathcal{E}$ & $[a\cos{t},b\sin{t}]$ \\
$\mathcal{L}(t)$ & line through $P(t)$ along $[P(t)-M]^\perp$ & \\
$Q_p,Q_{c},Q_{\theta}$ & pedal, contrapedal, rotated pedal feet & \\
$Q_{\mu}$ & linear interpolation of $Q_p,Q_{cp}$ & \\
$Q^*$ & intersection of pedal line with $\mathcal{L}(t)$ & \\
\hline
$\mathcal{E}_p$ & pedal curve of $\mathcal{E}$ wrt $M$ & locus of $Q_p$\\
$\mathcal{E}_{n}$ & negative pedal curve of $\mathcal{E}$ wrt $M$ & envelope of $\mathcal{L}(t)$ \\
$\mathcal{E}_{c}$ & contrapedal curve of $\mathcal{E}$ wrt $M$ & locus of $Q_{c}$  \\
$\mathcal{E}_{\theta}$ & rotated pedal curve of $\mathcal{E}$ wrt $M$ & locus of $Q_{\theta}$  \\
$\mathcal{E}_{\mu}$ & interpolated pedal curve of $\mathcal{E}$ wrt $M$ &  locus of $Q_{\mu}$ \\
$\mathcal{E}^*$ & hybrid pedal curve of $\mathcal{E}$ wrt $M$ & locus of $Q^*$ \\
$\mathcal{E}^\dagger$ & pseudo Talbot's curve of $\mathcal{E}$ wrt $M$ & locus of $Q^\dagger$ \\
\hline
$A$ & area of $\mathcal{E}$ & $\pi{a}{b}$ \\
$A_p,A_c,A_{\theta},A_\mu$ & areas of $\mathcal{E}_p,\mathcal{E}_c,\mathcal{E}_{\theta},\mathcal{E}_{\mu}$ & invariant for $M$ on a $\mathcal{K}_r$ \\
$A_n,A^*$ & areas of $\mathcal{E}_n,\mathcal{E}^*$ & invariant for $M$ on $\mathcal{E}$ \\
\hline
\end{tabular}
\caption{Symbols used.}
\label{tab:symbols}
\end{table}

\label{app:symbols}

\bibliographystyle{maa}
\bibliography{references,authors_rgk}

\end{document}